\documentclass[a4paper]{amsart}

\usepackage{amsmath,amssymb,amscd,graphics,enumerate,latexsym,verbatim,stmaryrd,graphicx,xcolor}
\usepackage[all]{xy}
\usepackage[linktoc=page,colorlinks,linkcolor={red!80!black},citecolor={red!80!black},urlcolor={blue!80!black}]{hyperref}

\usepackage[mathcal]{euscript}                               
\usepackage{mathptmx}
\DeclareSymbolFont{cmletters}{OML}{cmm}{m}{it}              
\DeclareSymbolFont{cmsymbols}{OMS}{cmsy}{m}{n}
\DeclareSymbolFont{cmlargesymbols}{OMX}{cmex}{m}{n}
\DeclareMathSymbol{\myjmath}{\mathord}{cmletters}{"7C}     \let\jmath\myjmath 
\DeclareMathSymbol{\myamalg}{\mathbin}{cmsymbols}{"71}     \let\amalg\myamalg
\DeclareMathSymbol{\mycoprod}{\mathop}{cmlargesymbols}{"60}
\DeclareMathSymbol{\myalpha}{\mathord}{cmletters}{"0B}     \let\alpha\myalpha 
\DeclareMathSymbol{\mybeta}{\mathord}{cmletters}{"0C}      \let\beta\mybeta
\DeclareMathSymbol{\mygamma}{\mathord}{cmletters}{"0D}     \let\gamma\mygamma
\DeclareMathSymbol{\mydelta}{\mathord}{cmletters}{"0E}     \let\delta\mydelta
\DeclareMathSymbol{\myepsilon}{\mathord}{cmletters}{"0F}   \let\epsilon\myepsilon
\DeclareMathSymbol{\myzeta}{\mathord}{cmletters}{"10}      \let\zeta\myzeta
\DeclareMathSymbol{\myeta}{\mathord}{cmletters}{"11}       \let\eta\myeta
\DeclareMathSymbol{\mytheta}{\mathord}{cmletters}{"12}     \let\theta\mytheta
\DeclareMathSymbol{\myiota}{\mathord}{cmletters}{"13}      \let\iota\myiota
\DeclareMathSymbol{\mykappa}{\mathord}{cmletters}{"14}     \let\kappa\mykappa
\DeclareMathSymbol{\mylambda}{\mathord}{cmletters}{"15}    \let\lambda\mylambda
\DeclareMathSymbol{\mymu}{\mathord}{cmletters}{"16}        \let\mu\mymu
\DeclareMathSymbol{\mynu}{\mathord}{cmletters}{"17}        \let\nu\mynu
\DeclareMathSymbol{\myxi}{\mathord}{cmletters}{"18}        \let\xi\myxi
\DeclareMathSymbol{\mypi}{\mathord}{cmletters}{"19}        \let\pi\mypi
\DeclareMathSymbol{\myrho}{\mathord}{cmletters}{"1A}       \let\rho\myrho
\DeclareMathSymbol{\mysigma}{\mathord}{cmletters}{"1B}     \let\sigma\mysigma
\DeclareMathSymbol{\mytau}{\mathord}{cmletters}{"1C}       \let\tau\mytau
\DeclareMathSymbol{\myupsilon}{\mathord}{cmletters}{"1D}   \let\upsilon\myupsilon
\DeclareMathSymbol{\myphi}{\mathord}{cmletters}{"1E}       \let\phi\myphi
\DeclareMathSymbol{\mychi}{\mathord}{cmletters}{"1F}       \let\chi\mychi
\DeclareMathSymbol{\mypsi}{\mathord}{cmletters}{"20}       \let\psi\mypsi
\DeclareMathSymbol{\myomega}{\mathord}{cmletters}{"21}     \let\omega\myomega
\DeclareMathSymbol{\myvarepsilon}{\mathord}{cmletters}{"22}\let\varepsilon\myvarepsilon
\DeclareMathSymbol{\myvartheta}{\mathord}{cmletters}{"23}  \let\vartheta\myvartheta
\DeclareMathSymbol{\myvarpi}{\mathord}{cmletters}{"24}     \let\varpi\myvarpi
\DeclareMathSymbol{\myvarrho}{\mathord}{cmletters}{"25}    \let\varrho\myvarrho
\DeclareMathSymbol{\myvarsigma}{\mathord}{cmletters}{"26}  \let\varsigma\myvarsigma
\DeclareMathSymbol{\myvarphi}{\mathord}{cmletters}{"27}    \let\varphi\myvarphi

\theoremstyle{plain}
\newtheorem{thm}{Theorem}[section]
\newtheorem{cor}[thm]{Corollary}
\newtheorem{lemma}[thm]{Lemma}
\newtheorem{prop}[thm]{Proposition}
\newtheorem*{thm*}{Theorem}
\newtheorem*{thmA}{Theorem A}

\theoremstyle{definition}
\newtheorem{df}[thm]{Definition}

\newtheorem*{question*}{Question}

\newtheorem{rem}[thm]{Remark}
\newtheorem{ex}[thm]{Example}

\DeclareMathOperator{\Spec}{Spec}

\DeclareMathOperator{\res}{res}
\DeclareMathOperator{\Hom}{Hom}

\DeclareMathOperator{\sign}{sign}
\DeclareMathOperator{\colim}{colim}

\DeclareMathOperator{\im}{im}

\DeclareMathOperator{\Mod}{{Mod}}

\def\0{{\bf 0}}
\def\A{{\mathbb A}}

\def\F{{\mathbb F}}

\def\P{{\mathbb P}}

\def\Z{{\mathbb Z}}

\def\cB{{\mathcal B}}
\def\cC{{\mathcal C}}
\def\cD{{\mathcal D}}

\def\cF{{\mathcal F}}

\def\cI{{\mathcal I}}
\def\cJ{{\mathcal J}}

\def\cL{{\mathcal L}}

\def\cO{{\mathcal O}}

\def\cR{{\mathcal R}}

\def\cT{{\mathcal T}}
\def\cU{{\mathcal U}}
\def\cV{{\mathcal V}}

\def\cX{{\mathcal X}}

\def\cZ{{\mathcal Z}}

\def\Fun{{\F_1}}
\def\Funsq{{\F_{1^2}}}

\def\int{\textup{int}}

\def\id{\textup{id}}
\def\1{\textbf{1}}

\def\rk{{\textup{rk}}}

\def\={\equiv}
\def\n={\equiv\hspace{-10,5pt}/\hspace{3,5pt}}

\def\bl{{\textup{bl}}}

\def\hom{{\textup{hom}}}

\def\tot{\textup{tot}}
\def\alt{\textup{alt}}
\def\hom{\textup{hom}}

\newdir{ >}{{}*!/-5pt/@^{(}}\newcommand{\arincl}[1]{\ar@{ >->}@<-0,0ex>#1} 

\newcommand{\gen}[1]{\langle #1 \rangle}

\newcommand{\bpquot}[2]{#1\!\sslash\!#2}
\newcommand{\bpgenquot}[2]{#1\!\sslash\!\gen{#2}}

\title{\v Cech cohomology over $\Funsq$}

\author{Jaret Flores}
\email{jarflores@gmail.com}
\address{}
\author{Oliver Lorscheid}
\email{oliver@impa.br}
\address{Instituto Nacional de Matem\'atica Pura e Aplicada, Estrada Dona Castorina 110, Rio de Janeiro, Brazil}
\author{Matt Szczesny}
\email{szczesny@math.bu.edu}
\address{Boston University, Dept. of Mathematics and Statistics, 111 Cummington Mall, Boston, MA 02215}

\begin{document}


\begin{abstract}
 In this text, we generalize Cech cohomology to sheaves $\mathcal F$ with values in blue $B$-modules where $B$ is a blueprint with $-1$. If $X$ is an object of the underlying site, then the cohomology sets $H^l(X,\mathcal F)$ turn out to be blue $B$-modules. For locally free $\mathcal O_X$-module $\mathcal F$ on a monoidal scheme $X$, we prove that $H^l(X,\mathcal F)^+=H^l(X^+,\mathcal F^+)$ where $X^+$ is the scheme associated with $X$ and $\mathcal F^+$ is the locally free $\mathcal O_{X^+}$-module associated with $\mathcal F$. 
 
 In an appendix, we show that the naive generalization of cohomology as a right derived functor is infinite-dimensional for the projective line over $\mathbb F_1$.
\end{abstract}

\maketitle

\section*{Introduction}
\label{section: introduction}

\noindent
 While many standard methods in algebraic geometry carry over readily to $\Fun$-geometry, other methods withstand a straightforward generalization since essential properties from usual algebraic geometry fail to be true or produce unusual results. 
 
 Sheaf cohomology with values in categories over $\Fun$ belongs to the latter class of theories. Though methods from homological algebra generalize without great difficulties to injective resolutions of sheaves on $\Fun$-schemes (see \cite{Deitmar11b}), the derived cohomology sets are larger than one would expect. For instance, the first cohomology set $H^1(X,\cO_{X})$ of the projective line $X=\P^1_\Fun$ over $\Fun$ is of infinite rank over $\Fun$, cf.\ Appendix \ref{appendix: cohomology of p1}.

 There have been some ad hoc observations for the projective line $\P^1_\Fun$ in \cite{Connes-Consani10a}, for which \v Cech cohomology works well as long as the chosen covering consists of at most two open sets. For larger coverings, however, it is not clear how to make sense of the alternating sums in the definition of \v Cech cohomology.\footnote{During the time of writing, Jaiung Jun has published his preprint \cite{Jun15} on \v Cech cohomology for semirings. His method of double complexes might be applicaple to the setting of this paper.}
 
 This problem resolves naturally for sheaves over $\Funsq$, since $\Funsq$ contains an additive inverse $-1$ of $1$, i.e.\ it bears a relation $1+(-1)=0$. This leads naturally to the theory of blueprints, which deals with multiplicative monoids that come together with certain additive relations that might be weaker than an addition.
  
 The aim of this paper is to define \v Cech cohomology for sheaves with values in blue $B$-modules where $B$ is a blueprint with $-1$, and to show that this leads to a meaningful theory.
 
 We calculate the cohomology of a monoidal scheme $X$ in terms of a comparison with the cohomology of their associated scheme $X^+$, which is also denoted by ``$X\otimes_\Fun\Z$'' in the literature. For this comparison, we assume the following mild technical assumption on an open covering $\{U_i\}_{i\in I}$ of $X$.
 
 \medskip\noindent\textbf{Hypothesis (H):} For all finite subsets $J\subset I$ of $\cI$, the restriction map
             \[
              \res_{U_J,U_I}:\cO_X(U_J)\to\cO_X(U_I)
             \]
             is injective.
 
 \medskip
 
 The following is Theorem \ref{thm: comparison of cech cohomology} of the main text.
 
\begin{thmA}
 Given a monoidal scheme $X$ over $B$ that admits a finite covering $\{U_i\}$ with Hypothesis (H) such that $\cO_X(U_i)$ are monoid blueprints over $B$. Then we have for every locally free sheaf $\cF$ on $X$ that
 \[
  H^l(X,\cF)^+ \ = \ H^l(X^+,\cF^+).
 \]
\end{thmA}
 
 Note that the class of monoidal schemes with a covering satisfying Hypothesis (H) contains, in particular, a model for every toric variety. Therefore the results of this paper might be helpful for calculations of sheaf cohomology for toric varieties, cf.\ Remark \ref{rem: cohomology for toric varieties}.

 In the first part of this paper, we define \v Cech cohomology for sheaves of blue $B$-modules on an arbitrary site. We choose this general formulation because it is applicable to arithmetic questions like the \'etale cohomology of the compactification $\overline{\Spec \Z}$ of the arithmetic line; see \cite{L14} for a model of $\overline{\Spec \Z}$ and some ideas towards such a theory.
 
 In the second part of this paper, we introduce the notion of monoidal schemes over a blueprint $B$, which extends the notion of monoidal schemes from $\Fun$ to any blueprint, and we discuss the notion of locally free sheaves. In a final section, we formulate and prove our main result Theorem A.
 
 Since there are several introduction to blueprints and blue schemes, we do not provide another one in this text, but provide the reader with a reference where this is necessary. As a general reference, we suggest the overview paper \cite{L13}. In particular, the reader will find the definition of a blueprint in section II.1.1, and the definition of a blue $B$-module in section II.6.1 of this paper.

\part{\v Cech cohomology over $\Funsq$}
\label{part: cech cohomology}

\section{Definition for a fixed covering}
\label{section: definition for a fixed covering}

\noindent
In this part of the paper, we consider a site $\cT$ and an object $X$ of this site. We assume that $\cT$ contains fibre products, so that we have a notion of covering families $\cU=\{U_i\}_{i\in \cI}$ of $X$. We will define \v Cech cohomology for $X$ with values in sheaves in blue $B$-modules where $B$ is a blueprint with $-1$, which can also be thought of as an $\Funsq$-algebra. 

Throughout this part of the paper, we fix the site $\cT$ and the object $X$. For this section, we also fix the covering family $\cU$ and aim for defining the \v Cech cohomology $H^l(X,\cF;\cU)$ w.r.t.\ $\cU$.

A blueprint with $-1$ is a blueprint that has an element $-1$ that satisfies the additive relation $1+(-1)\=0$. This element is necessarily unique, which means that there is a unique blueprint morphism $\Funsq\to B$ from
\[
 \Funsq \ = \ \bpgenquot{\{0,1,-1\}} {1+(-1)\=0} 
\]
to $B$. By multiplying the defining relation for $-1$ with an arbitrary element $a$ of $B$, we see that $-a=(-1)\cdot a$ is an \emph{additive inverse of $a$}, i.e.\ it satisfies the relation $a+(-a)\=0$.

Let $\Mod_B^\bl$ be the category of blue $B$-modules and $\cF$ a sheaf on $\cT$ with values in $\Mod_B^\bl$. Let $\cU=\{U_i\}_{i\in\cI}$ be a covering family of $X$ where $\cI$ is a totally ordered index set.

\begin{df}
 For $l\geq 0$, we denote by $\cI_l$ the family of all subsets $I$ of $\cI$ with cardinality $l+1$. For such a subset, we write $I=(i_0,\dotsb,i_l)$ if $I=\{i_0,\dotsc,i_l\}$ and $i_0<\dotsb<i_l$. We define
 \[
  U_I \ = \ U_{i_0}\times_X\dotsc\times_X U_{i_l} \qquad\text{and}\qquad \cF_I \ = \ \cF(U_I),
 \]
 which is a blue $B$-module. Given $I\in\cI_l$ and $k\in\{0,\dotsc,l\}$, we denote by $I^k$ the set $\{i_0,\dotsc,\widehat{i_k},\dotsc,i_l\}$. The canonical projection $U_I\to U_{I^k}$ onto all factors but $U_k$ defines a morphism
 \[
  \partial_{k,I}^{(l)}: \ \cF_{I^k} \ \longrightarrow \ \cF_I.
 \]
 If we define
 \[
  \cC^l \ = \ \prod_{I\in\cI_l}\cF_I
 \]
 the morphisms $\partial_{k,I}^{(l)}$ for varying $I$ define a morphism
 \[
  \partial_k^{(l)}: \ \cC^{l-1} \ \longrightarrow \ \cC^{l}
 \]
 for every $k=0,\dotsc,l$. The \emph{\v Cech complex of $\cU$ with values in $\cF$} is the cosimplicial blue $B$-module
 \[
  \cC^\bullet \ = \ \cC^\bullet(X,\cF;\cU) \ = \ \Biggl( \xymatrix{\cC^0 \ar@<0.5ex>[r]^{\partial_0^{(1)}}\ar@<-0.5ex>[r]_{\partial_1^{(1)}}  & \cC^1 \ar@<1ex>[r]^{\partial_0^{(2)}}\ar[r]\ar@<-1ex>[r]_{\partial_2^{(2)}}    & \cC^2 \ar@<1.5ex>[r]\ar@<0.5ex>[r]\ar@<-0.5ex>[r]\ar@<-1.5ex>[r]    & \cC^3 \quad\dotsb \quad} \Biggr).
 \]
\end{df}

\begin{rem} \label{rem: total ceh complex}
 In practice, the index set $\cI$ is often finite. Then the \v Cech complex is finite since $\cC^l$ is the empty product, i.e.\ $\cC^l=0$, if $l\geq\#\cI$. 

 This cosimplicial set is often called the \emph{ordered \v Cech complex} in literature, in contrast to the \emph{total \v Cech complex $\cC_\tot^\bullet$} with $\cC_{\tot}^l=\prod\cF_{\{i_0,\dotsc,i_l\}}$ where the product is taken over all elements $(i_0,\dotsc,i_l)\in\cI^{l+1}$ without any assumption on the ordering or distinctness of the $i_k$'s.
\end{rem}

\begin{df}
 Let $\cC^\bullet$ be a cosimplicial blue $B$-module. The \emph{set of $l$-cocycles of $\cC^\bullet$} is 
\[
 \cZ^l \ = \ \cZ^l(\cC^\bullet) \ = \ \biggl\{\ x\in\cC^l \ \biggl|\  \sum_{k=0}^{l+1} (-1)^k\partial_k^{(l+1)}(x) \= 0 \ \bigg\}
\]
which we consider as a full blue $B$-submodule of $\cC^l$, i.e.\ the pre-addition of $\cZ^l$ is the restriction of the pre-addition of $\cC^l$ to $\cZ^l$. The \emph{set of $l$-coboundaries} is
\[
 \cB^l \ = \ \cB^l(\cC^\bullet) \ = \ \biggl\{\ x\in\cC^l \ \biggl|\ \exists y=\sum_i y_i \in(\cC^{l-1})^+\text{ such that }x \= \sum_i\sum_{k=0}^{l} (-1)^k\partial_k^{(l)}(y_i) \ \bigg\},
\]
which is considered as a full blue $B$-submodule of $\cC^l$. For the case $l=0$, we use $\cC_{-1}=\{0\}$.

If $\cC^\bullet=\cC^\bullet(X,\cF;\cU)$, then we also write $\cZ^l(X,\cF;\cU) \ = \ \cZ^l(\cC^\bullet)$ and $\cB^l(X,\cF;\cU) \ = \ \cB^l(\cC^\bullet)$. In this case, we have
\[
 \cZ^l(X,\cF;\cU) \ = \ \biggl\{\ (a_I)\in \prod_{I\in\cI_l}\cF_I \ \biggl|\ \forall J\in\cI_{l+1},\quad \sum_{k=0}^{l+1} (-1)^k\partial_{k,I}^{(l+1)}(a_{J^k}) \= 0 \ \bigg\}
\]
and
\[
 \cB^l(X,\cF;\cU) \ = \ \biggl\{\ (a_I)\in \prod_{I\in\cI_l}\cF_I \ \biggl|\ \exists (b_J)\in \prod_{J\in\cI_{l-1}}\cF_J^+,\quad \forall I\in\cI_{l},\ \ a_I \= \sum_{k=0}^{l} (-1)^k\partial_{k,I}^{(l)}(b_{I^k})\ \bigg\}.
\]
where we define $\delta_{k,I}^{(l)}(b_{J})=\sum_j\delta_{k,I}^{(l)}(b_{J,j})$ for $b_{J}=\sum_j b_{J,j}\in\cF_J^+$ and $J=I^k$.
\end{df}

\begin{lemma}
 For every $l\geq0$, we have $\cB^l(\cC^\bullet)\subset \cZ^l(\cC^\bullet)$.
\end{lemma}

\begin{proof}
 Given $(a_i)\in\cB^l(\cC^\bullet)$, i.e.\ there is an element $(b_J)\in(\cC^{l-1})^+$ such that
 \[
  a_I \ \= \ \sum_{k=0}^{l} (-1)^k\partial_k^{(l)}(b_{I^k})
 \]
 for all $I\in\cI_l$, then we have for every $L\in\cI_{l+1}$
 \[
  \sum_k (-1)^k\partial_k^{(l+1)}(a_{L^k}) \ \= \ \sum_{k'\neq k} (-1)^k \ \partial_{k}^{l+1}\circ\partial_{k'}^l\,\Bigl((-1)^{k'+\epsilon}b_{L^{k,k'}}\Bigr)
 \]
 where $\epsilon=0$ if $k'<k$ and $\epsilon=1$ if $k'>k$, and $L^{k,k'}=L-\{k,k'\}$. Since $\partial_{k}^{l+1}\circ\partial_{k'}^l=\partial_{k'}^{l+1}\circ\partial_{k}^l$, the above sum equals
 \[
  \sum_{k'< k} (-1)^{k+k'}\ \partial_{k}^{l+1}\circ\partial_{k'}^l\,\bigl(b_{L^{k,k'}}\bigr) \quad + \quad \sum_{k< k'} (-1)^{k+k'+1}\ \partial_{k}^{l+1}\circ\partial_{k'}^l\,\bigl(b_{L^{k,k'}}\bigr) \quad \= \quad 0,
 \]
 which shows that $(a_I)\in\cZ^l(\cC^\bullet)$.
\end{proof}

\begin{df}
 The \emph{$l$-th \v Cech cohomology of $X$ w.r.t\ $\cU$ and with values in $\cF$} is defined as the quotient
 \[
  H^l(\cC^\bullet) \ = \ \cZ^l(\cC^\bullet)\ / \ \cB^l(\cC^\bullet)
 \]
 of blue $B$-modules. If $\cC^\bullet=\cC^\bullet(X,\cF;\cU)$, then we also write $H^l(X,\cF;\cU) \ = \ H^l(\cC^\bullet)$.
\end{df}

Recall that a morphism $\Psi: \cC^\bullet\to\cD^\bullet$ of cosimplicial blue $B$-modules is a collection of morphisms $\psi_l:\cC^l\to\cD^l$ of blue $B$-modules for all $l\geq0$ that commute with the respective coboundary maps $\partial_k^{(l)}$ of $\cC^\bullet$ and $\cD^\bullet$, i.e.\ $\partial_{k}^{(l)}\circ\psi_{l-1}=\psi_l\circ\partial_{k}^{(l)}$ for all $l\geq 0$ and $0\leq k\leq l$.

\begin{lemma}\label{lemma: morphisms induce maps between cohomology}
 Let $\Psi: \cC^\bullet\to\cD^\bullet$ be a morphism of cosimplicial blue $B$-modules. Then 
 \[
  \psi_l(\cZ^l(\cC^\bullet))\subset\cZ^l(\cD^\bullet) \qquad \text{and} \qquad \psi_l(\cB^l(\cC^\bullet))\subset\cB^l(\cD^\bullet). 
 \]
 Consequently, $\Psi$ induces a morphism
 \[
  H^l(\cC^\bullet) \ = \ \cZ^l(\cC^\bullet) \ / \ \cB^l(\cC^\bullet) \quad \longrightarrow \quad \cZ^l(\cD^\bullet) \ / \ \cB^l(\cD^\bullet) \ = \ H^l(\cD^\bullet)
 \]
 for every $l\geq0$.
\end{lemma}

\begin{proof}
 Let $x\in\cZ^l(\cC^\bullet)$, i.e.\ $\sum(-1)^k\partial_k^{(l+1)}(x)\=0$. Then $\psi_l(x)\in\cD^l$ satisfies
 \[
  \sum_{k=0}^{l+1} \ (-1)^k \ \partial_k^{(l+1)}(\psi(x)) \quad \= \quad \sum_{k=0}^{l+1} \ (-1)^k \ \psi_{l+1}\partial_k^{(l+1)}(x) \quad \= \quad \psi_{l+1}(0) \quad \= \quad 0.
 \]
 This shows that $\psi(x)\in\cZ^l(\cD^\bullet)$. Let $x\in\cB^l(\cC^\bullet)$, i.e.\ there exists an $y \in\cC^{l-1}$ with $x\=\sum(-1)^k\partial_k^{(l)}(y)$. Then we have
 \[
  \psi_l(x) \quad \= \quad \sum(-1)^k\psi_l\Bigl(\partial_k^{(l)}(y)\Bigr) \quad \= \quad \sum(-1)^k\partial_k^{(l)}\bigl(\psi_{l-1}(y)\bigl),   
 \]
 which shows that $\psi_l(x)\in\cB^l(\cD^\bullet)$.
\end{proof}

Next, we prove that the \v Cech cohomology w.r.t.\ $\cU$ does not depend on the ordering of the index set $\cI$. Note that the definition of the \v Cech complex $\cC^\bullet$ is independent of the ordering of $\cI$.

\begin{prop}
 For $l\geq0$, the subsets $\cB^l$ and $\cZ^l$ of $\cC^l$ are independent of the ordering of $\cI$. Consequently, $H^l(X,\cF;\cU)$ does not depend on the ordering of $\cI$.
\end{prop}

\begin{proof}
 The usual argument works in this context: we show that $\cC^\bullet$ is isomorphic to the \emph{alternating \v Cech complex $\cC_\alt^\bullet$} as a cosimplicial blue $B$-module. The blue $B$-modules $\cC^l_\alt$ are defined as all elements $a_I$ of $\cC^l_\tot$ that satisfy 
 \[
  a_I \ = \ 0
 \]
 if $I=(i_0,\dotsc,i_l)$ with $i_k=i_{k'}$ for some $k\neq k'$, and
 \[
   a_{\sigma I} \ = \ \sign(\sigma) \, a_{I}   
 \]
 for a permutation $\sigma\in S_{l+1}$ and $\sigma I=(i_{\sigma(0)},\dotsc,i_{\sigma(l)})$. This defines a cosimplicial subset $\cC^\bullet_\alt $ of $\cC^\bullet_\tot$. Consider the following morphisms of blue $B$-modules
 \[
  \begin{array}{cccc}
   \pi: & \cC^l_\alt             & \longrightarrow & \cC^l. \\
          & (a_I)_{I\in\cI^{l+1}} & \longmapsto     & (a_I)_{I\in\cI_{l}}
  \end{array}
 \]
 and
 \[
  \begin{array}{cccc}
   \iota: & \cC^l             & \longrightarrow & \cC^l_\alt \\
          & (a_I)_{I\in\cI_l} & \longmapsto     & (\widetilde{a_I})_{I\in\cI^{l+1}}
  \end{array}
 \]
 with $\widetilde{a_I}=\sign(\sigma)a_{\sigma I}$ if $\sigma I\in\cI_l$ and $\widetilde{a_I}=0$ if $I=(i_0,\dotsc,i_l)$ with $i_k=i_{k'}$ for some $k\neq k'$. As in the usual case of \v Cech cohomology with values in abelian categories, it is easily verified that $\iota$ and $\pi$ are mutually inverse isomorphisms.

 If $\widetilde\cI$ is the index set $\cI$ with a different ordering and $\tilde\pi:\cC^\bullet_\alt\to\cC^\bullet$ is the isomorphism with respect to this ordering, then the automorphism $\tilde\pi\circ\iota:\cC^\bullet\to\cC^\bullet$ sends the set $\cZ^l$ of $l$-cocycles w.r.t.\ to the ordering of $\cI$ to the set $\widetilde\cZ^l$ of $l$-coboundaries w.r.t.\ the ordering of $\widetilde\cI$. More precisely, $\tilde\pi\circ\iota$ sends $a_I$ to $\sign(\sigma)a_{\sigma I}$ where $\sigma$ is the permutation such that $\sigma I$ is ordered w.r.t.\ to the ordering of $\cI$. Since $B$ is with $-1$, we see that $\widetilde\cZ^l=\cZ^l$.

 Similarly, $\tilde\pi\circ\iota$ restricts to an automorphism of $\cB^l$. This shows the claim of the proposition.
\end{proof}


\section{Refinements}
\label{section: refinements}

\noindent
In this section, we show that the \v Cech cohomology $H^l(X,\cF;\cU)$ is functorial in refinements, so that we form the colimit $H^l(X,\cF)=\colim H^l(X,\cF;\cU)$, which does not depend on the choice of a covering family of $X$ anymore.

\begin{df}
 A \emph{refinement of a covering family $\cU=\{U_i\}_{i\in\cI}$} is a covering family $\cV=\{V_j\}_{j\in\cJ}$ together with a map $\varphi:\cJ\to\cI$ and a morphism $\varphi_j:V_j\to U_{\varphi(i)}$ for every $j\in\cJ$. We write $\Phi:\cV\to \cU$ for such a refinement.
\end{df}

Given a refinement $\Phi:\cV\to \cU$ of $\cU$, we get induced maps $\varphi:\cJ_l\to\cI_l$ that send $J=\{j_0,\dotsc,j_l\}$ to $\varphi(J)=\{\varphi(j_0),\dotsc,\varphi(j_l)\}$ and morphisms
\[
 \varphi_J: \quad V_J \ = \ V_{j_0}\times_X \dotsb \times_X V_{j_l} \quad \longrightarrow \quad U_{\varphi(j_0)}\times_X \dotsb \times_X U_{\varphi(j_l)} \ = \ U_{\varphi(J)}
\]
for every $J\in\cJ_l$ and $l\geq 0$. This defines, in turn, a morphism $\psi_l:\cC^l(X,\cF;\cU)\to\cC^l(X,\cF;\cV)$ for every $l\geq0$. The morphisms $\psi_l$ commute with the respective coboundary morphisms $\partial_k^{(l)}$ of $\cC^\bullet(X,\cF;\cV)$ and $\cC^\bullet(X,\cF;\cU)$. Thus $\Phi:\cV\to\cU$ induces a morphism $\Psi:\cC^\bullet(X,\cF;\cU)\to\cC^\bullet(X,\cF;\cV)$ of cosimplicial blue $B$-modules, which maps cocycles to cocycles and coboundaries to coboundaries. This means that we get a morphism
\[
 \Psi: \quad H^l(X,\cF;\cU) \quad \longrightarrow \quad H^l(X,\cF;\cV) 
\]
from the \v Cech cohomology w.r.t.\ $\cU$ to the \v Cech cohomology w.r.t.\ $\cV$.

\begin{df}
 The \emph{\v Cech cohomology of $X$ with values in $\cF$} is defined as the colimit
 \[
  H^l(X,\cF) \ = \ \colim\ H^l(X,\cF;\cU)
 \]
 over the system of all covering families $\cU$ of $X$ together with all refinements $\Phi:\cV\to\cU$ of covering families.
\end{df}


\part{Cohomology of monoidal schemes}
\label{part: monoidal schemes}

\section{Monoidal schemes over a blueprint}
\label{section: monoidal schemes}

\noindent
Monoidal schemes a.k.a.\ monoid schemes a.k.a.\ $\Fun$-schemes (in the sense of Deitmar, \cite{Deitmar05}, or To\"en and Vaqui\'e, \cite{Toen-Vaquie09}) form the core of $\Fun$-geometry in the sense that they appear as a natural subclass in every approach towards $\Fun$-schemes.

In this section, we introduce monoidal schemes over a blueprint $B$ as certain blue schemes over $B$. Note that refer to the notion of blue schemes from \cite{L15}, which can be seen as an improvement of the original definition in terms of prime ideals, as contained in \cite{L13}. If $B$ happens to be a global blueprint, e.g.\ a monoid, a ring or a blue field, then both definitions give rise to an equivalent theory. In this case, one can also adopt the viewpoint of To\"en and Vaqui\'e in \cite{Toen-Vaquie09}, which yields yet another theory in general.

To start with, we adapt the concept of a semigroup ring to the context of blueprints. By a \emph{monoid}, we mean a commutative and associative semigroup with neutral element $1$ and absorbing element $0$. All monoids will be written multiplicatively. A monoid morphism is a multiplicative map that sends $1$ to $1$ and $0$ to $0$.

Let $B=\bpquot{A}{\cR}$ be a blueprint and $M$ a monoid. The \emph{monoid blueprint of $M$ over $B$} is the blueprint $B[M]=\bpquot{A_M}{\cR_M}$ that is defined as follows. The monoid $A_M$ is the smash product $A\wedge M$, which is the quotient of $A\times M$ by the equivalence relation that is generated by the relations $(a,0)\sim (b,0)$ and $(0,m)\sim(0,n)$ with $a,b\in A$ and $m,n\in M$. The pre-addition $\cR$ is generated by the set of additive relations 
\[
 \bigl\{ \ \sum(a_i,1)\=\sum(b_j,1) \ \bigl| \ \sum a_j\=\sum b_j\text{ in }B \ \bigr\}.
\]
Note that $B[M]$ has the universal property that any pair of a blueprint morphism $B\to C$ and a monoid morphism $M\to C$ extends uniquely to a blueprint morphism $B[M]\to C$. Note further that as a blue $B$-module, $B[M]$ is isomorphic to \mbox{$\bigvee_{m\in M-\{0\}} B\cdot m$}.

\begin{df}
 Let $B$ be a blueprint. A \emph{monoidal scheme over $B$} is a blue scheme $X$ that has an open affine covering $\{U_i\}_{i\in\cI}$ such that 
 \begin{enumerate}
  \item for every $i\in\cI$, there is a monoid $M_i$ and an isomorphism $\cO_X(U_i)\simeq B[M_i]$ of blueprints;
  \item for every $i,j\in\cI$, the intersection $U_i\cap U_j$ is covered by affine opens of the form $V_{i,j,k}=\Spec B[N_{i,j,k}]$ for some monoids $N_{i,j,k}$ such that the restriction map $\res:B[M_i]\to B[N_{i,j,k}]$ is the localization at some multiplicative subset $S_{i,k}$ of $M_i$, i.e.\ $N_{i,j,k}=S^{-1}_{i,k} M_i$; and the same holds true for the restriction map $\res:B[M_j]\to B[N_{i,j,k}]$.
 \end{enumerate}
\end{df}

\begin{rem}
 Note that a monoidal scheme over $\Fun$ is nothing else than a monoidal scheme in the usual sense. One can extend the method from \cite{Deitmar08} to show that $X^+_\Z$ is a toric variety over the ring $B^+_\Z$ if $X$ is connected separated integral torsion-free monoidal scheme of finite type over $B$.
\end{rem}

\begin{prop}\label{prop: monoidal schemes are defined over f1}
 Let X be a blue scheme over $B$. Then $X$ is monoidal over $B$ if and only if there is a monoidal scheme $X_\Fun$ over $\Fun$ such that $X$ is isomorphic to $X_\Fun\times_{\Spec\Fun}\Spec B$.
\end{prop}

\begin{proof}
 Since $B[M]\simeq\Fun[M]\otimes_\Fun B$, it is clear that the base extension $X_\Fun\times_{\Spec\Fun}\Spec B$ of a monoidal scheme $X_\Fun$ to $B$ is monoidal over $B$.

 To prove the other direction of the equivalence, assume that $X$ has a covering $\{U_i\}$ with $U_i=\Spec B[M_i]$ for certain monoids $M_i$. The pairwise intersections $U_i\cap U_j$ have coverings $\{V_{i,j,k}\}$ with $V_{i,j,k}=\Spec B[N_{i,j,k}]$ where each monoid $N_{i,j,k}$ is a localization of both $M_i$ and $M_j$, i.e.\ 
 \[
  B[N_{i,j,k}] \quad = \quad B[S_{i,k}^{-1}M_i]\quad = \quad B[S_{j,k}^{-1}M_j]
 \]
 for certain multiplicative subsets $S_{i,k}$ of $M_i$ and $S_{j,k}$ of $M_j$. If $\cD$ is the diagram of all $U_i$ and $V_{i,j,k}$ together with the inclusions $V_{i,j,k}\to U_i$ and $V_{i,j,k}\to U_j$, then $X$ is the colimit of $\cD$.

 We define the affine monoidal schemes $U_{i,\Fun}=\Spec\Fun[M_i]$ and $V_{i,j,k}=\Spec\Fun[N_{i,j,k,\Fun}]$. The colimit of resulting diagram $\cD_\Fun$ defines a blue scheme $X_\Fun$, which is monoidal over $\Fun$ since $\{U_{i,\Fun}\}$ is a covering of $X_\Fun$. It is clear from the construction that $X\simeq X_\Fun\times_{\Spec\Fun}\Spec B$. This finishes the proof of the proposition.
\end{proof}

Recall that a blue scheme $X$ is \emph{separated over $\Fun$} if the diagonal morphism $\Delta: X\to X\times X$ is a closed immersion. An important consequence is that the intersection of two affine subschemes of a separated blue scheme is affine.

\begin{cor}\label{cor: intersections of opens in monoidal schemes}
 Let $X_\Fun$ be a separated monoidal scheme over $\Fun$ and $X_B=X\otimes_\Fun B$ its base extension to $B$. Consider two open affine subsets $U_1$ and $ U_2$ of $X_B$ such that $\cO_{X_B}(U_1)$ and $\cO_{X_B}(U_2)$ are monoid blueprints over $B$. Then $\cO_{X_B}(U_1\cap U_2)$ is a monoid blueprint over $B$.
\end{cor}

\begin{proof}
 Let $M_1$ and $M_2$ be monoids such that $\cO_{X_B}(U_i)\simeq B[M_i]$ for $i=1,2$. By Proposition \ref{prop: monoidal schemes are defined over f1}, there are open affine subschemes $V_1$ and $V_2$ of $X$ such that $\cO_X(V_i)\simeq\Fun[M_i]$ for $i=1,2$. We have
 \[
  U_1 \ \cap \ U_2 \ = \ U_1 \, \times_{X_B} \, U_2 \ = \ \bigl( V_1 \, \times_X \, V_2 \bigr) \, \otimes_\Fun \, B \ = \ \bigl( V_1 \, \cap \ V_2 \bigr) \, \otimes_\Fun \ B.
 \]
 Since $X$ is separated over $\Fun$, the intersection $V_0=V_1\cap V_2$ is affine. By \cite[Thm.\ 30]{Vezzani12}, there is a monoid $M_0$ that is a localization of both $M_1$ and $M_2$ such that $\cO_X(V_0)\simeq\Fun[M_0]$. Since the intersection $U_0=U_1\cap U_2$ is isomorphic to the base extension of $V_0$ to $B$, we have $\cO_{X_B}(U_0)\simeq B[M_0]$. This proves the corollary.
\end{proof}

For monoidal schemes over blue fields we can conclude the following.

\begin{prop}\label{prop: monoidal schemes over blue fields}
 If $B$ is a blue field and $X$ is a monoidal scheme over $B$, then every open subset $U$ of $X$ has an open affine covering $\{U_i\}$ with $U_i=\Spec B[M_i]$ for certain monoids $M_i$.
\end{prop}

\begin{proof}
 Let $U$ be an open subset of $X$ and $\{V_j\}$ an open affine covering with $\cO_X(V_j)=B[N_j]$. Then the intersections $U\cap V_j$ can be covered by subsets $W_{j,k}$ such that $\cO_X(W_{j,k})$ is a localization of $B[N_j]$. Since $B$ is a blue field, we have for every multiplicative subset $S$ of $B[N_j]$ that 
 \[
  S^{-1}B[N_j] \ = \ T^{-1}B[N_j] \ = \ B[T^{-1}N_j] \qquad \text{where}\qquad T \ = \ S\,\cap\, \{ \,1\cdot a\,|\, a\in N_j\,\}.
 \]
 Thus $U$ is covered by the $W_{j,k}$ and $\cO_X(W_{j,k})$ are monoid blueprints over $B$, which proves the proposition.
\end{proof}

\begin{ex}
 The proposition is not true over an arbitrary blueprint $B$ since the localizations of $B$ are in general not monoid blueprints. For instance consider the integers $B=\Z$. Then $\Spec\Z$ is a monoidal scheme over $\Z$, but every proper open subset is of the form $\Spec \Z[d^{-1}]$ for some integer $d\geq 2$, and $\Z[d^{-1}]$ is not a monoid blueprint over $\Z$ since we have
 \[
  \underbrace{d^{-1}+\dotsb+d^{-1}}_{d\text{ times}} \ \= \ 1.
 \]
 Note further that even if $B$ is a blue field, not every open subset $U$ of a monoidal scheme $X$ over $B$ satisfies that it is isomorphic to the spectrum of a monoid blueprint over $B$. For instance consider $X=\Spec(B\times B)=\Spec B\amalg\Spec B$, which is monoidal over $B$ since it is covered by two copies of $\Spec B$. However, $B\times B$ is not a monoid blueprint over $B$ since it contains the additive relation
 \[
  (1,0)+(0,1) \ \= \ (1,1).
 \]
\end{ex}


\section{Locally free sheaves}
\label{section: locally free sheaves}

\noindent
Let $B$ be a blueprint and $M$ a blue $B$-module.

\begin{df}
 A \emph{basis for $M$} is a subset $\beta$ of $M$ such that 
 \begin{enumerate}
  \item for all $m\in M$, there are elements $a_1,\dotsc,a_r\in B$ and pairwise distinct elements $b_1,\dotsc,b_r\in \beta$ such that
        \[
         m \ \= \ \sum_{i=1}^r \ a_i b_i.
        \]
  \item If 
        \[
         \sum_{i=r}^r \ a_i b_i \ \= \ \sum_{j=s}^r \ a'_j b'_j
        \]
        for $a_1,\dotsc,a_r,a'_1,\dotsc,a'_s\in B$ and pairwise distinct elements $b_1,\dotsc,b_r,b'_1,\dotsc,b'_s\in M$, then $a_1=\dotsb=a_r=a'_1=\dotsb=a'_s=0$.
 \end{enumerate}
 A blue $B$-module is \emph{freely generated} if it has a basis. A blue $B$-module is \emph{free} if is isomorphic to $\bigvee_{b\in\beta}B\cdot b$ for a subset $\beta$ of $B$.
\end{df}

Note that $\beta$ is a basis for the free blue module $\bigvee_{b\in\beta}B\cdot b$. Thus a free module is freely generated. The larger class of freely generated modules can be classified as follows.

\begin{lemma}
 Let $M$ be a blue $B$-module with basis $\beta$.
 \begin{enumerate}
  \item There is a unique isomorphism from $M$ onto a blue submodule of
        \[
         \bigoplus_{b\in\beta} \ B \cdot b \quad = \quad \bigl\{ \ (m_b)\in\prod_{b\in\beta}B \cdot b \ \bigl| \ m_b=0\text{ for all but finitely many }b\ \bigr\} 
        \]
        that maps $b\in\beta$ to $1\cdot b$. Conversely, any blue $B$-submodule of $\bigoplus_{b\in\beta}B\cdot b$ that contains $\beta$ is freely generated by $\beta$.
  \item Any two bases of $M$ have the same cardinality.
 \end{enumerate}
\end{lemma}

\begin{proof}
 Since every $m\in M$ is a unique linear combination $m\=\sum m_b b$ of the basis elements $b\in\beta$ where all but finitely many $m_b\in B$ are $0$, the only possible morphism $\Phi:M\to\bigoplus_{b\in\beta}B\cdot b$ sends $m$ to $(m_b b)$. Since there are no additive relations between the different basis elements by \eqref{part2} of the definition of a basis, the map $\Phi$ is indeed a morphism of blue $B$-modules. It is clearly injective and thus defines an isomorphism onto its image. The latter claim of \eqref{part1} of the lemma is obvious.
 
 The embedding $\Phi:M\to \bigoplus_{b\in\beta}B\cdot b$ defines an isomorphism $\Phi^+:M^+\to \bigoplus_{b\in\beta}B^+\cdot b$ of blue $B^+$-modules, which have the same basis $\beta$. Thus we obtain an isomorphism $\Phi_\Z^+:M^+_\Z\to \bigoplus_{b\in\beta}B_\Z^+\cdot b$ of free $B_\Z^+$-modules with basis $\beta$. Since any two bases of a free module over a ring have the same cardinality, we obtain the same result for the freely generated blue $B$-module $M$.
\end{proof}

\begin{df}
 Let $M$ be a freely generated blue $B$-module with basis $\beta$. The \emph{rank $\rk_B M$ of $M$ over $B$} is defined as the cardinality of $\beta$.
\end{df}

Let $X$ be a blue scheme over $B$ and $\beta$ a (possibly infinite) set of cardinality $r$. In this part of the paper, a sheaf on $X$ is a sheaf on the small Zariski site of $X$.

\begin{df}
 A \emph{locally free sheaf of rank $r$} on $X$ is a sheaf $\cF$ on $X$ in blue $B$-modules that has an open affine covering $\{U_i\}_{i\in\cI}$ with the following properties:
 \begin{enumerate}
  \item\label{part1} if $B_i=\cO_X(U_i)$, then for every $i\in\cI$,
        \[
         \cF(U_i) \quad \simeq \quad \bigvee_{b\in\beta} B_i\cdot b\,;
        \]
  \item\label{part2} for every $i\in\cI$ and every open subset $V$ of $U_i$ with $\cO_X(V)=S_V^{-1}B_i$ for some multiplicative subset $S_V$ of $B_i$, there is an isomorphism $\cF(V)\simeq\bigvee_{b\in\beta} S_V^{-1}B_i\cdot b$ such that the restriction map 
        \[
         \res_{U_i,V}:\quad \bigvee_{b\in\beta} B_i\cdot b \quad \longrightarrow \quad \bigvee_{b\in\beta} S_V^{-1}B_i\cdot b
        \]
        corresponds to the localization of each component $B_i\cdot b$ at $S_V$.
 \end{enumerate}
 We call a covering $\{U_i\}$ of $X$ that satisfies properties \eqref{part1} and \eqref{part2} a \emph{trivialization of $\cF$}.
\end{df}

Note that the \emph{localizations} $V$ of the $U_i$, i.e.\ open subsets of the form $V=\Spec S_V^{-1} B_i$ for some multiplicative subset $S_V$ of $B_i$, form a basis for the topology of $X$. Thus a locally free sheaf $\cF$ is uniquely determined by a trivialization $\{U_i\}$ together with the restriction maps to subsets of the form $V=\Spec S_V^{-1} B_i$.

\begin{rem}
 There is an obvious notion of a quasi-coherent sheaf on $X$ (cf.\ \cite{CLS12} for the case of monoidal schemes). Property \eqref{part2} is automatically satisfied if $\cF$ is quasi-coherent. In other words, a quasi-coherent sheaf is locally free if and only if there are a set $\beta$ and an open affine covering $\{U_i\}$ of $X$ such that $\cF(U_i) \simeq \bigvee_{b\in\beta} B_i\cdot b$ for all $i$.
\end{rem}

\begin{ex}
 The sheaf $\bigvee_{b\in\beta}\cO_X$ that sends an open subset $U$ of $X$ to $\bigvee_{b\in\beta}\cO_X(U)$, together with the obvious restriction maps, is locally free of rank $r=\#\beta$. It is called the \emph{trivial locally free sheaf of rank $r$}.
\end{ex}

We construct the base extension of a locally free sheaf to rings. Let $\cF$ be a locally free sheaf on $X$ of rank $r$ and $\cU$ the family of all open affine subsets of $X$ such that $\cF(U)\simeq\bigvee_{b\in\beta} \cO_X(U)\cdot b$ together with all inclusion maps. By properties \eqref{part1} and \eqref{part2}, $X$ is the colimit of $\cU$. If $\cU_\Z^+$ denotes the family of all $U_\Z^+$ for $U$ in $\cU$ and all inclusion maps $U_\Z^+\to V_\Z^+$ whenever $U\to V$ is in $\cU$, then $X_\Z^+$ is the colimit of $\cU_\Z^+$.

We define $\cF_\Z^+(U_\Z^+)=\bigl(\cF(U)\bigr)_\Z^+$ for all $U$ in $\cU$ and we obtain restriction morphisms $\res:\cF(U)_\Z^+\to\cF(V)_\Z^+$ for every inclusion $V\to U$ in $\cU$. Since localizations commute with base extensions to rings, i.e.\ $(S^{-1}B)^+_\Z =S^{-1}(B_\Z^+)$, the values $\cF_\Z^+(U_\Z^+)$ for $U$ in $\cU$ glue together to a uniquely determined sheaf $\cF_\Z^+$ on $X_\Z^+$. Since
\[
 \Bigl( \ \bigvee_{b\in\beta} B_U \cdot b \Bigr)^+_\Z \quad = \quad \bigoplus_{b\in\beta} \ B_{U,\Z}^+ \cdot b,
\]
the sheaf $\cF_\Z^+$ is locally free on $X$ as a sheaf with values in $B_\Z^+$-modules.


\section{{\v C}ech cohomology of monoidal schemes}
\label{section: cech cohomology of monoidal schemes}

\noindent
In this section, we prove the comparison result for the cohomology of locally free sheaves on monoidal schemes with the cohomology of its base extension to rings.

Let $B$ be a blueprint with $-1$ and $X$ a monoidal scheme over $B$. Let $\beta$ be a set of cardinality $r$ and $\cF$ a locally free sheaf of rank $r$ on $X$. A trivialization $\cU=\{U_i\}_{i\in\cI}$ of $\cF$ is \emph{finite} if $\cI$ is a finite set. It is \emph{monoidal} if the coordinate blueprints $B_i=\cO_X(U_i)$ are monoid blueprints of the form $B[M_i]$ over $B$. 

We employ the notation from Part \ref{part: cech cohomology} of the paper. We assume that $\cI$ is totally ordered and denote by $\cI_l$ the set of cardinality $l+1$-subsets $I$ of $\cI$, which inherits an ordering from $\cI$. We write $I=(i_0,\dotsc,i_l)$ if $I=\{i_0,\dotsc,i_l\}$ and $i_0<\dotsb<i_l$. For $I\in\cI_l$, we define 
\[
 U_I \ = \ \bigcap_{i\in I} \ U_{i} \ , \qquad B_I \ = \ \cO_X(U_I) \qquad \text{and} \qquad \cF_I \ = \ \cF(U_I).
\]
Let $\cC^l=\prod_{I\in\cI_l}\cF_I$ and $\cC^\bullet=\cC^\bullet(X,\cF;\cU)$ the \v Cech complex of $X$ w.r.t.\ $\cU$ and values in $\cF$. We denote the coboundary maps as usual by $\partial_k^{(l)}:\cC^{l-1}\to\cC^l$. We state the following hypothesis on $X$ and $\cU=\{U_i\}_{i\in\cI}$.

\medskip\noindent\textbf{Hypothesis (H):} For all finite subsets $J\subset I$ of $\cI$, the restriction map
             \[
              \res_{U_J,U_I}:\cO_X(U_J)\to\cO_X(U_I)
             \]
             is injective.

\begin{rem}
Recall that a blueprint $B$ is \emph{integral} if every non-zero element $a\in B$ acts injectively on $B$ by multiplication. A blue scheme is \emph{integral} if the coordinate blueprint of every open affine subscheme $U$ of $X$ is integral. If $X$ is integral, then (H) is satisfied for all open affine coverings $\cU$ of $X$.
\end{rem}

Since $B$ is with $-1$, we have that $B^+$ is a ring, i.e.\ $B^+=B^+_\Z$. Similarly, $X^+=X^+_\Z$ is a scheme over $B^+$ and $\cF^+$ is a locally free sheaf on $X^+$ in $B^+$-modules. Defining $\cU^+$ as the collection of all affine opens $U_i^+$ of $X^+$, we obtain a trivialization of $\cF^+$ and can form the \v Cech complex $\cC^\bullet(X^+,\cF^+;\cU^+)$ of $X^+$ w.r.t.\ $\cU^+$ and values in $\cF^+$. Then the subsets $\cZ^l(X^+,\cF^+;\cU^+)$ and $\cB^l(X^+,\cF^+;\cU^+)$ are $B^+$-modules, and so is $H^l(X^+,\cF^+;\cU^+)$.

There is a canonical morphism $\cC^\bullet(X,\cF;\cU) \longrightarrow \cC^\bullet(X^+,\cF^+;\cU^+)$ of cosimplicial blue $B$-modules, which is injective in each degree since all blue $B$-modules $\cC^l(X,\cF;\cU)$ are with $-1$. Thus we can consider $\cZ^l(X,\cF;\cU)$ as a subset of $\cZ^l(X^+,\cF^+;\cU^+)$ and $\cB^l(X,\cF;\cU)$ as a subset of $\cB^l(X^+,\cF^+;\cU^+)$ for every $l\geq 0$. This induces a morphism 
\[
 H^l(X,\cF;\cU) \quad \longrightarrow \quad H^l(X^+,\cF^+;\cU^+)
\]
of blue $B$-modules.

\begin{thm}\label{thm: comparison for a fixed covering}
 Given a monoidal scheme $X$ over $B$, a locally free sheaf $\cF$ on $X$ and a finite monoidal trivialization $\cU=\{U_i\}_{i\in\cI}$ of $\cF$ that satisfies Hypothesis (H). Then 
 \[
  \cZ^l(X,\cF,\cU)^+ \ = \ \cZ^l(X^+,\cF^+,\cU^+) \qquad \text{and} \qquad \cB^l(X,\cF,\cU)^+ \ = \ \cB^l(X^+,\cF^+,\cU^+) 
 \]
 for every $l\geq0$. Consequently, we have
 \[
  H^l(X,\cF,\cU)^+ \ = \ H^l(X^+,\cF^+,\cU^+).
 \]
\end{thm}

\begin{proof}
We will establish the following two lemmas in order to prove Theorem \ref{thm: comparison for a fixed covering}. In the proofs of these lemmas, we will make use of the usual \v Cech chain complex
\begin{multline*}
  \cC^0(X^+,\cF^+;\cU^+) \quad \stackrel{d^1}\longrightarrow \quad \cC^1(X^+,\cF^+;\cU^+) \quad \longrightarrow \quad \dotsb \\ 
  \dotsb \quad \longrightarrow \quad \cC^{l-1}(X^+,\cF^+;\cU^+) \quad \stackrel{d^l}\longrightarrow \quad \cC^l(X^+,\cF^+;\cU^+) \quad \longrightarrow \quad \dotsb
\end{multline*}
where the differentials $d^l=\sum_{i=0}^l (-1)^i \partial_k^{(l)}$ are the alternating sums of the respective restriction maps.

\begin{lemma}\label{lemma: base extension of coboundary blueprints}
 $\cB^l(X,\cF,\cU)^+ \ = \ \cB^l(X^+,\cF^+,\cU^+)$.
\end{lemma}

\begin{proof}
 Since $\cU$ is finite, a set of generators for $\cB^l(X^+,\cF^+,\cU^+)$ is given by the images of the vectors $x_{a,b,J}=(0,\dotsc,0,a\cdot b,0,\dotsc,0)$ with $a\in B_J$ and $b\in\beta$. The image of such a vector is of the form $(d^l(x_{a,b,J})_I)_{I\in\cI_l}$. Since $x_{a,b,J}$ has only one non-trivial component, we have $d^l(x_{a,b,J})_I=\partial_k^{(l)}(x_{a,b,J})_I$ for some $k$. Therefore, the image of $x$ in $\cC^{l,+}$ is contained in $\cC^l$. Since $\cB^l(X,\cF,\cU)=\cB^l(X^+,\cF^+,\cU^+)\cap\cC^l$, the lemma follows. 
\end{proof}

\begin{lemma}\label{lemma: base extension of cocycle blueprints}
 $\cZ^l(X,\cF,\cU)^+ \ = \ \cZ^l(X^+,\cF^+,\cU^+)$.
\end{lemma}

\begin{proof}
 Let $B_\eta=\colim B_I$ be the colimit of the blueprints $B_I$ for finite subsets $I$ of $\cI$. By Hypothesis (H), the canonical inclusions $B_I\to B_\eta$ are injective for all finite $I\subset\cI$. Since $\partial_k^{(l)}$ extends to a $B_{I^k}^+$-linear map
 \[
  \partial_k^{(l),+}: \quad \cF_{I^k}^+ \ \simeq \ \bigoplus_{b\in\beta} \ B_{I^k}^+\cdot b \quad \longrightarrow \quad \bigoplus_{b\in\beta} \ B_I^+\cdot b \ \simeq \ \cF_I^+,
 \]
 the \v Cech chain complex 
 \[
  \cC^0(X^+,\cF^+;\cU^+) \quad \stackrel{d^1}\longrightarrow \quad \cC^1(X^+,\cF^+;\cU^+) \quad \longrightarrow \quad \dotsb
 \]
 with $\cC^l(X^+,\cF^+;\cU^+)=\prod_{i\in\cI_l}\bigoplus_{b\in\beta}B_I^+\cdot b$ defines a chain complex
 \[
  \cC^{0,+}_\eta \quad \stackrel{d^1}\longrightarrow \quad \cC^{1,+}_\eta \quad \longrightarrow \quad \dotsb
 \]
 with $\cC^{l,+}_\eta=\prod_{i\in\cI_l}\bigoplus_{b\in\beta}B_\eta^+\cdot b$. This chain complex is the \v Cech chain complex of the affine scheme $X_\eta^+=\Spec B_\eta^+$ w.r.t.\ the covering $\cU_\eta^+=\{U_{i,\eta}^+\}_{i\in\cI}$ where $U_{i,\eta}^+=X_\eta^+$ and with values in the locally free sheaf $\cF_\eta^+$ associated to the $B_\eta^+$-module $F_\eta^+=\colim \cF^+(U_I^+)$ that is the colimit over all finite $I\subset\cI$.

  Since the cohomology of coherent sheaves on affine schemes is concentrated in degree $0$, we have 
 \begin{align*}
  \cZ^0(X_\eta^+,\cF^+_\eta;\cU_\eta^+) \quad &= \quad F_\eta^+                              &&\text{and}       \\
  \cZ^l(X_\eta^+,\cF^+_\eta;\cU_\eta^+) \quad &= \quad \cB^l(X_\eta^+,\cF^+_\eta;\cU_\eta^+) &&\text{for}\quad l>0.
 \end{align*}

 Since $F_\eta^+=\bigl(\bigvee_{b\in\beta} B_\eta\cdot b\bigl)^+$ is generated by elements in $\cZ^0(X,\cF;\cU)=\bigvee_{b\in\beta} B\cdot b$ as a blue $B_\eta$-module, the claim of the lemma follows for $l=0$. 
 
 For $l>0$, we can apply Lemma \ref{lemma: base extension of coboundary blueprints} to $X_\eta=\Spec B_\eta$, the locally free sheaf associated with $F_\eta=\colim \cO_X(U_I)$ and $\cU_\eta=\{U_{i,\eta}\}_{i\in\cI}$ with $U_{i,\eta}=X_\eta$ and get
 \[
   \cB^l(X_\eta^+,\cF^+_\eta;\cU_\eta^+) \quad = \quad \cB^l(X_\eta,\cF_\eta;\cU_\eta)^+. 
 \]
 Since
 \[
  \cB^l(X_\eta,\cF_\eta;\cU_\eta) \quad \subset \quad \cZ^l(X_\eta,\cF_\eta;\cU_\eta) \quad \subset \quad \cZ^l(X_\eta^+,\cF^+_\eta;\cU_\eta^+),
 \]
 we conclude that $\cZ^l(X_\eta,\cF_\eta;\cU_\eta)^+=\cZ^l(X_\eta^+,\cF^+_\eta;\cU_\eta^+)$. Therefore 
 \begin{align*}
  \cZ^l(X,\cF;\cU)^+ \quad &= \quad \Bigl( \ \cC^l(X,\cF;\cU) \quad \cap \quad \cZ^l(X_\eta,\cF_\eta;\cU_\eta) \ \Bigr)^+ \\
                     \quad &= \quad \cC^l(X,\cF;\cU)^+ \quad \cap \quad \cZ^l(X_\eta,\cF_\eta;\cU_\eta)^+                \\
                     \quad &= \quad \cC^l(X^+,\cF^+;\cU^+) \quad \cap \quad \cZ^l(X_\eta^+,\cF_\eta^+;\cU_\eta^+)  \\
                     \quad &= \quad \cZ^l(X^+,\cF^+;\cU^+)
 \end{align*}
 as desired.
\end{proof}

Since taking quotients commutes with the base extension to rings, we have that
\begin{multline*}
 H^l(X,\cF;\cU)^+ \ = \ \cZ^l(X,\cF;\cU)^+ / \cB^l(X,\cF;\cU)^+ \ = \\ \cZ^l(X^+,\cF^+;\cU^+) / \cB^l(X^+,\cF^+;\cU^+) \ = \ H^l(X^+,\cF^+;\cU^+),
\end{multline*}
which proves Theorem \ref{thm: comparison for a fixed covering}.
\end{proof}

\begin{thm}\label{thm: comparison of cech cohomology}
 Given a monoidal scheme $X$ over $B$ that admits a finite covering $\{U_i\}$ with Hypothesis (H) such that $\cO_X(U_i)$ are monoid blueprints over $B$. Then we have for every locally free sheaf $\cF$ on $X$ that
 \[
  H^l(X,\cF)^+ \ = \ H^l(X^+,\cF^+).
 \]
\end{thm}

\begin{proof}
 Let $\cU$ be a covering of $X$ with Hypothesis (H) and $\cF$ a locally free sheaf on $X$. Then there is a finite refinement $\cV$ of $\cU$ that satisfies all conditions of Theorem \ref{thm: comparison for a fixed covering}. Since we can choose $\cU$ itself arbitrary fine, the coverings $\cV$ that satisfy the hypotheses of Theorem \ref{thm: comparison for a fixed covering} form a cofinal system in the category of all finite coverings of $X$ together with refinements. Since $X$ is quasi-compact, the $\cV$ are cofinal in the category of all coverings of $X$.

 Therefore the colimit of the cohomology blueprints $H^l(X,\cF;\cV)$ over all coverings $\cV$ that satisfy Theorem \ref{thm: comparison for a fixed covering} equals $H^l(X,\cF)$. For the same reasons, the colimit of the cohomology groups $H^l(X^+,\cF^+,\cV^+)$ over all such $\cV$ equals $H^l(X^+,\cF^+)$. Since $(-)^+$ commutes with filtered colimits, this establishes the claim of the theorem. 
\end{proof}

\begin{ex}[Line bundles on projective space] 
 Let $B$ be a blueprint with $-1$ and $\cO(d)$ the twisted sheaf on $\P^n_B$. If $d\geq0$, then the cohomology $H^\ast(\P^{n,+}_B,\cO(d)^+)$ is concentrated in degree $0$. Therefore $H^0(\P^n_B,\cO(d))$ is the only non-trivial cohomology of $\P^n_B$ with values in $\cO(d)$. It is clear that $H^0(\P^n_B,\cO(d))$ equals the blue $B$-module of global sections of $\cO(d)$, which is a free $B$-module of rank $\rk\; H^0(\P^{n,+}_B,\cO(d)^+)$.

 For $-n\leq d\leq -1$, the cohomology $H^\ast(\P^n_B,\cO(d))$ is trivial. If $d\leq -n-1$, then the cohomology $H^\ast(\P^{n,+}_B,\cO(d)^+)$ is concentrated in degree $n$. Therefore $H^n(\P^n_B,\cO(d))$ is the only non-trivial cohomology of $X$ with values in $\cO(d)$. If $\cU=\{U_i\}_{i\in\cI}$ is the canonical atlas of $\P^n_B$, then we have $H^n(\P^n_B,\cO(d))=H^n(\P^n_B,\cO(d);\cU)$ by comparison with the compatible situation for $\P^{n,+}_B$ and the canonical covering $\cU^+$. Therefore we have $\cZ^l(\P^n_B,\cO(d))=\cO(d)(U_\cI)$, and $\cB^l(\P^n_B,\cO(d))$ is generated by the images $d(x_{a,b,J})\in\cO(d)(U_\cI)$ (cf.\ the proof of Lemma \ref{lemma: base extension of coboundary blueprints}). Therefore $H^n(\P^n_B,\cO(d))$ is a free blue $B$-module of rank $\rk\; H^n(\P^{n,+}_B,\cO(d)^+)$.
\end{ex}

Also in more complicated examples, we found that the cohomology blueprints are free over the base blueprint. Therefore we pose the following problem.

\begin{question*}
 Let $B$ be a blueprint with $-1$ and $X$ a quasi-compact monoidal scheme over $B$ that admits an open affine covering satisfying Hypothesis (H). Is it true that $H^l(X,\cF)$ is a free blue $B$-module for every locally free sheaf $\cF$?
\end{question*}

\begin{rem}[Sheaf cohomology for toric varieties] \label{rem: cohomology for toric varieties}
 We conlcude this text with the following remark on possible applications to the computation of sheaf cohomology for toric varieties.
 
 Every toric variety $\cX$ over the ring $B^+$ admits a monoidal model $X$ over $B$, i.e.\ a monoidal scheme $X$ over $B$ such that $\cX\simeq X^+$ as a $B^+$-scheme. The maximal open affine covering of $X$ satisfies Hypothesis (H) since the restriction maps correpond to inclusions of subsemigroups of the ambient character lattice of the toric variety.
 
 Since the \v Cech cohomology for monoidal schemes is amenable to explicit calculation due to their rigid structure, Theorem \ref{thm: comparison of cech cohomology} yields an application for calculations of sheaf cohomology over toric varieties.
 
 The drawback is, however, that only a very limited class of locally free sheaves over toric varieties can be defined over a monoidal model. Namely, the rigid structure of the wedge product implies that every locally free sheaf $\cF$ on a monoidal scheme $X$ over a blueprint $B$ decomposes into the wedge product $\bigvee \cL_i$ of line bundles.
 
 This means that the only locally free sheaves of toric varieties for which our methods apply are (direct sums of) line bundles. There exists an algorithm to calculate the cohomology of toric line bundles, as conjectured in \cite{BJRR10} and proven independently in \cite{Roschy-Rahn10} and \cite{Jow11}. The method of this algorithm seems to be quite different from the perspective of our text, but it would be interesting to understand the precise relationship. 
\end{rem}


\appendix

\section{Cohomology of {$\P^1$} via injective resolutions}
\label{appendix: cohomology of p1}

\noindent
In this section, we mimic the methods of homological algebra and injective resolutions to calculate the cohomology $H^i_\hom(X,\cO_X)$ of the projective line $X=\P^1_\Fun$. While $H_\hom^0(X,\cO_X)$ equals the global sections of $\cO_X$, it turns out that $H_\hom^1(X,\cO_X)$ is of infinite rank over $\Fun$. 

Note that the following calculations apply also to the projective line over $\Funsq$, which shows that $H^i_\hom(X,\cF)$ differs from the cohomology blueprints $H^i(X,\cF)$, as considered in the main text of this paper.

Deitmar has given in \cite{Deitmar11b} a rigorous treatment of cohomology via injective resolutions for sheaves in so called \emph{belian} categories. This applies, in particular, to sheaves on $\P^1_\Fun$ in pointed $\Fun$-modules (also known as pointed $\Fun$-sets). Note that the general hypotheses of \cite{Deitmar11b} are not satisfied by the category of blue $B$-modules.

To emphasize that we abandon any additive structure in the discussion that follows, we avoid mentioning blueprints, but employ the language of monoids and monoidal schemes.

Let $A=\Fun[T]$ be the coordinate monoid of $\A^1_\Fun$. All of the $A$-modules in the following are pointed $A$-modules (following the terminology of \cite{Deitmar11b}), and we denote the base point generally by $\ast$. Let $F=\{T^i\}_{i\geq0}\cup\{\ast\}$ be the free module over $A$ of rank $1$, $I=\{T^i\}_{i\in\Z}\cup\{\ast\}$ and $J=\{T^i\}_{i<0}\cup\{\ast\}$. Then both $I$ and $J$ are injective $A$-modules. Let $G=\Fun[T^{\pm1}]$ be the ``quotient monoid'' of $A$. Then the corresponding localizations of $I$ and $J$ are $I$ itself resp.\ $0=\{\ast\}$, which are both injective $G$-modules.

The topological space of $X=\P^1_\Fun$ has three points; namely, two closed points $x_1,x_2$ and one generic point $x_0$. It can be covered by two opens $U_i=\{x_0,x_i\}$ ($i=1,2$), which are both isomorphic to $\A^1_\Fun$ and which intersect in $U_0=\{x_0\}$. The coordinate monoids of these opens are respectively $\cO_X(U_1)\simeq \cO_X(U_2)\simeq A$ and $\cO_X(U_0)\simeq G$, where $\cO_X$ is the structure sheaf of $X$.

We define the injective sheaf $\cI_0$ over $X$ by $\cI_0(U_i)=I$ for $i=0,1,2$ together with the identity maps $\id:I\to I$ as restriction maps. We define the injective sheaf $\cI_1$ over $X$ by $\cI_0(U_i)=J$ for $i=1,2$ and $I_1(U_0)=0$ together with the trivial maps $0:J\to 0$ as restriction maps.

It is easily seen that the structure sheaf $\cO_X$ of $X$ has an injective resolution of the form
\[
 0 \ \longrightarrow \ \cO_X \ \longrightarrow \ \cI_0 \ \longrightarrow \ \cI_1 \ \longrightarrow \ 0.
\]
Taking stalks at $x_1$ or at $x_2$ yields the exact sequence 
\[
 0 \ \longrightarrow \ F \ \longrightarrow \ I \ \longrightarrow \ J \ \longrightarrow \ 0
\]
of $A$-modules. Talking stalks at $x_0$ yields the exact sequence
\[
 0 \ \longrightarrow \ I \ \longrightarrow \ I \ \longrightarrow \ 0 \ \longrightarrow \ 0
\]
of $H$-modules. 

The next step is to apply $\Hom(\cO_X,-)$ to the given injective resolution of $\cO_X$. A morphism $\varphi:\cO_X\to \cI_0$ is determined by the image of $\varphi_{x_0}(T^0)\in \cI_{1,x_0}=I$ of $T^0\in \cI_{0,x_0}=I$. Thus $\Hom(\cO_X,\cI_0)\simeq I$ (as $A$-module, or even $H$-module). 

A morphism $\psi:\cO_X\to \cI_1$ is given by two $A$-module maps \[\psi_i: \cO_{X,x_i}=F \to J=\cI_{1,x_i}\] ($i=1,2$), which do not have to satisfy any relation since the restriction maps of $\cI_1$ are trivial. Thus $\Hom(\cO_X,\cI_1)\simeq J\times J$ (as $A$-module). Note that, a priori, these homomorphism sets are merely $\Fun$-modules, but the richer structure as $A$-modules makes it easier to study the induced morphism $\Phi:\Hom(\cO_X,\cI_0)\to \Hom(\cO_X,\cI_1)$, which is the only non-trivial map in the complex
\[
 0 \ \longrightarrow \ \Hom(\cO_X,\cI_0) \ \stackrel\Phi\longrightarrow \ \Hom(\cO_X,\cI_1) \ \longrightarrow \ 0.
\]
We define the cohomology groups $H_\hom^i(\P^1_\Fun,\cO_X)$ as the cohomology groups of this complex.

The kernel of $\Phi$ consists of the trivial morphism and the morphism $\varphi:\cO_X\to\cI_0$ that is characterized by $\varphi_{x_0}(T^0)=T^0$. Thus $H_\hom^0(\P^1_\Fun,\cO_X)=\{\ast,\varphi\}$ is an $1$-dimensional $\Fun$-vector space, in accordance with the analogous result for sheaf cohomology of $\P^1$ over a ring. 

The image of $\Phi$ are morphisms $\psi:\cO_X\to\cI_1$ such that either $\psi_1$ or $\psi_2$ is trivial. Thus $\im\Phi=J\vee J\subset J\times J$ (as $A$-modules). Consequently $H_\hom^1(\P^1,\cO_X)=(J\times J)/(J\vee J)$ is an infinite-dimensional $\Fun$-vector space. This result is not at all in coherence with the situation over a ring where $H_\hom^1(\P^1,\cO_X)=0$.

\begin{rem}
 The above calculation can also be used to calculate $H_\hom^i(\P^1,\cO(n))$ for the twists $\cO(n)$ of the structure sheaf, which yields the expected outcome for $H_\hom^0$, namely, an $\Fun$-vector space of dimension $n+1$ if $n\geq1$ and $0$ if $n<0$, but which yields, again, an infinite-dimensional $\Fun$-vector space $H_\hom^1(\P^1,\cO(n))$.
\end{rem}

\begin{rem} 
 As explained to the second author by Anton Deitmar, this does not contradict Theorem 2.7.1 in \cite{Deitmar11b}, which implies that the rank of the cohomology over $\Fun$ is at most the rank of the corresponding cohomology over $\Z$. The reason is that the base extension of the twisted sheaf $\cO(n)$ to $\Z$ (in the sense of \cite{Deitmar11b}) is not the twisted sheaf on the projective line over $\Z$, but a sheaf on $\P^1_\Z$ that is not of finite type. 

 To explain, the definition of the base extension of a sheaf $\cF$ on an $\Fun$-scheme $X$ to the associated scheme $X_\Z$ in \cite{Deitmar11b} is the pullback $\pi^\ast\cF$ along the base extension map $\pi: X_\Z\to X$, not tensored with the structure sheaf of $X$. This differs from the sheaf $\cF^+_\Z$ considered in this text.
\end{rem}


\bibliographystyle{plain}

\begin{thebibliography}{10}

\bibitem{BJRR10}
Ralph Blumenhagen, Benjamin Jurke, Thorsten Rahn and Helmut Roschy
\newblock Cohomology of line bundles: a computational algorithm.
\newblock {\em J. Math. Phys.}, 51(10), 2010.

\bibitem{CLS12}
Chenghao Chu, Oliver Lorscheid, and Rekha Santhanam.
\newblock Sheaves and {$K$}-theory for {$\mathbb F_1$}-schemes.
\newblock {\em Adv. Math.}, 229(4):2239--2286, 2012.

\bibitem{Connes-Consani10a}
Alain Connes and Caterina Consani.
\newblock Schemes over {$\mathbb F_1$} and zeta functions.
\newblock {\em Compos. Math.}, 146(6):1383--1415, 2010.

\bibitem{Deitmar05}
Anton Deitmar.
\newblock Schemes over {$\mathbb F\sb 1$}.
\newblock In {\em Number fields and function fields---two parallel worlds},
  volume 239 of {\em Progr. Math.}, pages 87--100. Birkh\"auser Boston, Boston,
  MA, 2005.

\bibitem{Deitmar08}
Anton Deitmar.
\newblock {$\mathbb F_1$}-schemes and toric varieties.
\newblock {\em Beitr\"age Algebra Geom.}, 49(2):517--525, 2008.

\bibitem{Deitmar11b}
Anton Deitmar.
\newblock Belian categories.
\newblock Far East J. Math. Sci. 70(1):1--46, 2012.

\bibitem{Jow11}
Shin-Yao Jow.
\newblock Cohomology of toric line bundles via simplicial Alexander duality.
\newblock {\em J. Math. Phys.}, 52(3), 2011.

\bibitem{Jun15}
Jaiung Jun.
\newblock {C}ech cohomology of semiring schemes.
\newblock Preprint, \href{http://arxiv.org/pdf/1503.01389}{arXiv:1503.01389}, 2015.

\bibitem{L13}
Oliver Lorscheid.
\newblock A blueprinted view on {$\mathbb{F}_1$}-geometry.
\newblock Preprint, \href{http://arxiv.org/pdf/1301.0083}{arXiv:1301.0083}, to appear in the ECM monograph \emph{Absolute Arithemetic and {$\mathbb{F}_1$}-geometry}, 2013.

\bibitem{L14}
Oliver Lorscheid.
\newblock Blueprints--towards absolute arithmetic?
\newblock J. Number Theory 144:408--421, 2014. 

\bibitem{L15}
Oliver Lorscheid.
\newblock Scheme theoretic tropicalization.
\newblock Preprint, \href{http://arxiv.org/pdf/1508.07949}{arXiv:1508.07949}, 2015.

\bibitem{Roschy-Rahn10}
Helmut Roschy and Thorsten Rahn.
\newblock Cohomology of line bundles: proof of the algorithm.
\newblock {\em J. Math. Phys.}, 51(10), 2010.

\bibitem{Toen-Vaquie09}
Bertrand To{\"e}n and Michel Vaqui{\'e}.
\newblock Au-dessous de {${\rm Spec}\,\mathbb Z$}.
\newblock {\em J. K-Theory}, 3(3):437--500, 2009.

\bibitem{Vezzani12}
Alberto Vezzani.
\newblock Deitmar's versus {T}o\"en-{V}aqui\'e's schemes over {$\mathbb{F}_1$}.
\newblock {\em Math. Z.}, 271(3-4):911--926, 2012.

\end{thebibliography}

\end{document}